\documentclass[12pt]{amsart}

\usepackage{amsmath,amssymb}
\usepackage{amsfonts}
\usepackage{amsthm}
\usepackage{latexsym}
\usepackage{graphicx}
\usepackage{epsfig}

\newtheorem{theorem}{Theorem}[section]
\newtheorem{conj}{Conjecture}[section]
\newtheorem{lemma}{Lemma}[section]
\newtheorem{remark}{Remark}[section]
\newtheorem{coro}{Corollary}[section]

\newtheorem{prop}{Proposition}[section]
\newtheorem{definition}{Definition}[section]
\newtheorem{prob}{Problem}[section]

\def\RR{{\mathrm R}}
\def\WW{{\mathrm W}}

\def\Rc{{\mathrm {Rc}}}

\def\SS{{\mathrm S}}
\def\KK{{\mathrm K}}

\def\diag{\mathrm{Diag}}
\def\id{\mathrm{Id}}
\def\det{\mathrm{det}}

\begin{document}

\title{Four-manifolds of Pinched Sectional Curvature}

\author{Xiaodong Cao}

\address{Department of Mathematics,
 Cornell University, Ithaca, NY 14853-4201}
\email{cao@math.cornell.edu}

\author{Hung Tran}

\address{Department of Mathematics and Statistics,
Texas Tech University, Lubbock, TX 79413}
\email{hung.tran@ttu.edu}

\renewcommand{\subjclassname}{%
  \textup{2010} Mathematics Subject Classification}

\date{ \today}
\begin{abstract} In this paper, we study closed four-dimensional manifolds. In particular, we show that, under various new pinching curvature conditions (for example, the sectional curvature is no more than $ \frac{5}{6}$ of the smallest Ricci eigenvalue) then the manifold is definite. If restricting to a metric with harmonic Weyl tensor, then it must be self-dual or anti-self-dual under the same conditions. Similarly, if restricting to an Einstein metric, then it must be either the complex projective space with its Fubini-Study metric, the round sphere or their quotients. Furthermore, we also classify Einstein manifolds with positive intersection form and an upper bound on the sectional curvature. 
\end{abstract}
\maketitle

\section{Introduction}
A fundamental theme in mathematics is to study the relation between the geometry and topology. The geometry is normally realized by some curvature conditions while the topology would involve invariants such as Betti numbers, Euler characteristic, or Hirzebruch signature. One of many famous questions by H. Hopf in that theme is the following. 
\begin{conj}
	(Hopf) $\mathbb{S}^2\times \mathbb{S}^2$ does not admit a Riemannian metric with positive sectional curvature.
\end{conj}
The intuition is that $\mathbb{S}^2\times \mathbb{S}^2$ is characterized by its topological invariants and some might be an obstruction to the existence of a metric with positive sectional curvature. While the conjecture is still open, it is observed by R. Bettiol that there is a metric on  $\mathbb{S}^2\times \mathbb{S}^2$ with positive biorthogonal curvature \cite{Bettiol14}. At any point $x$, the biorthogonal (sectional) curvature of a plane $P\in T_x(M)$ is defined as
\begin{equation*}
\KK^{\perp}(x, P)=\frac{\KK(x, P)+\KK(x, P^{\perp})}{2},
\end{equation*}
where $P^\perp$ is the orthogonal plane to $P$ and $\KK$ is the sectional curvature. Clearly, positive sectional curvature implies positive biorthogonal curvature. Throughout this paper, our conditions on the biorthogonal curvature are all point-wise, when there is no confusion, we will omit the point $x$ and simply denote as $\KK^{\perp}(P)$.\\

In this paper, we will show that if $\mathbb{S}^2\times \mathbb{S}^2$ admits a metric with positive biorthogonal curvature, then, at each point, such curvature must be polarized across different planes. That is, the maximum will be relatively high in comparison with the minimum. Indeed, that follows from a more general theorem which partly determines the topology of a manifold if assuming
one of several curvature assumptions. Precisely, in addition to the pinched sectional curvature condition discussed above, we also consider inequalities between sectional curvature and scalar curvature $\SS$ and Ricci curvature $\Rc$. In that direction, our results can also be considered as progress in addressing the following problem by S. T. Yau. 
\begin{prob}
	(S.T. Yau \cite[Problem 12]{Yauopen93}) The famous pinching problem says that on a compact simply connected manifold if $\KK_{\min}>\frac{1}{4}\KK_{\max}>0$ then the manifold is homeomorphic to a sphere. If we replace $\KK_{\max}$ by normalized scalar curvature, can we deduce similar pinching results? 
\end{prob}

Indeed, we have the following statement.   
\begin{theorem}
\label{definitetheorem}
Let $(M,g)$ be a closed four-dimensional manifold with positive scalar curvature and let $\lambda_1>0$ be its first eigenvalue of the Laplacian on functions. Suppose one of the following conditions holds:
\begin{enumerate}
\item $\KK^\perp\leq \frac{\SS(2\SS+9\lambda_1)}{12(\SS+3\lambda_1)};$
\item There exists a  $k>0$, such that $\Rc\geq k$ and $\KK^\perp \leq \frac{5}{6}k$.
\item $\KK^\perp\geq \frac{\SS^2}{24(\SS+3\lambda_1)};$
\item $\KK^\perp_{min}\geq \frac{\SS}{2(2\SS+9\lambda_1)} \KK^\perp_{max};$
\end{enumerate}
  Then $M$ has definite intersection form.
\end{theorem}
\begin{remark}
Part $(3)$ is independently observed by R. Diogenes, E. Ribeiro Jr, and E. Rufino \cite{DRF18}. Also see  \cite{DR18} for a related work. 
\end{remark}
\begin{remark} Parts $(1), (2), (3)$ of Theorem \ref{definitetheorem} also hold when replacing the biorthogonal sectional curvature $\KK^\perp$ by the regular $\KK$.
\end{remark}
\begin{remark} In comparison, if one assumes a stronger condition, $\KK^\perp_{min}\geq \frac{1}{4} \KK^\perp_{max}$ or $\KK^\perp \geq \frac{\SS}{24}$ or  $\KK^\perp \leq \frac{\SS}{6}$, then it is observed that the manifold has nonnegative isotropic curvature; see \cite{Seaman91}, \cite{CR14}. 
\end{remark}

\begin{remark} In dimension four, the Hodge star operator induces a natural decomposition of the vector bundle of 2-forms into spaces of self-dual and anti-self-dual two forms. A manifold is said to have definite intersection form if either the space of harmonic self-dual or anti-self-dual 2-forms is trivial (zero dimension). For a precise definition, see Section \ref{prem}. Furthermore, by quoting the topological classification of M. Freedman \cite{Freedman82}, one can determine the homeomorphic type of the manifold admitting such a metric.
\end{remark} 
In particular, one consequence is the following. 
\begin{coro}\label{S2timeS2}
	$\mathbb{S}^2\times \mathbb{S}^2$ does not admit a metric satisfying any aforementioned condition. 
\end{coro}

It turns out that restricting to special metrics yields more precise results. First, we consider a manifold admitting a metric with harmonic Weyl curvature. This is a generalization of the Einstein equation and has been studied intensively; see, for example, \cite{Derd83, mw93, tran16weyl}. In that setting, we have the following.
\begin{theorem}
	\label{selfdualtheorem}
	Let $(M,g)$ be a closed four-dimensional manifold with harmonic Weyl tensor and positive scalar curvature. Let $\lambda>0$ be its first eigenvalue of the Laplacian on functions. Suppose one of the following conditions holds:
	\begin{enumerate}
		\item $\KK^\perp\leq \frac{\SS(2\SS+9\lambda)}{12(\SS+3\lambda)};$
		\item $\Rc\geq k>0$ and $\KK^\perp \leq \frac{5}{6}k$;
		\item $\KK^\perp\geq \frac{\SS^2}{24(\SS+3\lambda_1)};$
		\item $\KK^\perp_{min}\geq \frac{\SS}{2(2\SS+9\lambda)} \KK^\perp_{max}$. 
	\end{enumerate}
	Then the manifold must be either self-dual or anti-self-dual.
\end{theorem}
Again, the Hodge star operator gives rise to a natural decomposition of the Weyl tensor into self-dual and anti-self-dual Weyl curvature. The conclusion of the theorem means that either one of them is vanishing. In addition, by combining with a result of A. Derdzinski \cite{Derd88}, we obtain a classification.
\begin{coro} \label{classselfdual}
Let $(M,g)$ be a closed four-dimensional manifold with harmonic Weyl tensor and positive scalar curvature. Suppose that $g$ is analytic and the biorthognal sectional curvature satisfies one of the above conditions then $(M,g)$ is either locally conformally flat or homothetically isometric to $\mathbb{CP}^2$ with its Study-Fubini metric or $\mathbb{S}^4$ with the round metric or its quotient.  
\end{coro}
\begin{remark} Our results improve earlier results obtained by E. Costa and E. Ribeiro \cite{CR14}. After completing this paper, it was brought to our attention that Part (3) of Theorem \ref{selfdualtheorem} was first obtained by Ribeiro \cite{Ribeiro16}. 
\end{remark}
Next, we turn our attention to the setting of an Einstein manifold. It is noted that, in that case, the biorthogonal sectional curvature is identical to the regular sectional curvature. The application of Theorem \ref{selfdualtheorem} yields the followings.  	
\begin{coro}
	\label{Eupperbound}
	Let $(M,g)$ be a smooth compact oriented four-dimensional Einstein manifold with $\Rc=g$. Suppose one of the following conditions holds:
	\begin{enumerate}
		\item $\KK\leq \frac{4(8+9\lambda)}{12(4+3\lambda)};$
		\item $\KK \leq \frac{5}{6}$;
		\item $\KK\geq \frac{2}{3(4+3\lambda_1)};$
		\item $\KK_{min}\geq \frac{2}{8+9\lambda} \KK_{max}$. 
	\end{enumerate}
	then $(M,g)$ is homothetically isometric to $\mathbb{CP}^2$ with its Study-Fubini metric or $\mathbb{S}^4$ with its round metric or its quotient.    
\end{coro}
\begin{remark} It was brought to our attention that some parts of Corollary \ref{Eupperbound} are independently obtained by Q. Cui and L. Sun \cite{cs18} recently. 
\end{remark}
Here, we would like to bring readers' attention to a famous folklore conjecture regarding Einstein structures. 
\begin{conj}\label{Epositive} A simply connected Einstein four manifold with positive scalar curvature and non-negative sectional curvature must be either $\mathbb{S}^4$ with its round metric, $\mathbb{CP}^2$ with its Fubini-Study metric, $\mathbb{S}^2\times \mathbb{S}^2$ with its product metric. 
\end{conj}

This conjecture has attracted tremendous interests but proves to be quite obdurate. Nevertheless, there have been various contributions, see \cite{berger61, gl99, tachibana, yangdg00, costa04, brendle10einstein, cao14einstein, caotran4} and the references therein. It is noted that, for the normalization $\Rc=g$, $\KK\geq 0$ implies $\KK\leq 1$ (which is equivalent to 4-nonnegative curvature operator). Thus, it is of great interest to study Einstein structures with an upper bound on sectional curvature. Indeed, Corollary \ref{Eupperbound} improves the earlier results in \cite{caotran4}.  

Furthermore, the manifold in Theorem \ref{definitetheorem} must either have positive intersection form or have its second Betti number vanishing. Interestingly, M. Gursky and C. LeBrun were able to solve Conjecture \ref{Epositive} in case the manifold has positive intersection form\cite{gl99}. 
We observe that the lower bound in Gursky-Lebrun's theorem can be replaced by an upper bound. 
\begin{theorem}
\label{Epositiveintersectionthm}
Let $(M, g)$ be a smooth closed oriented Einstein four-manifold with positive intersection form and  $\Rc=g$. Suppose that \[\KK\leq 1,\]
then $(M,g)$ is homothetically isometric to $\mathbb{CP}_2$ with its standard Fubini-Study metric.   
\end{theorem}

Here is a sketch of the proof. The main idea is to apply the Bochner techniques in a manner similar to \cite{tran16weyl, caotran4}. That is, we proceed by contradiction. Suppose the desired conclusion is not true then we construct a function with zero average and its Laplacian controllable. The zero average allows us to obtain 
an equality involving the first eigenvalue of the Laplacian on functions. The curvature assumption then allows us to estimate zero-order terms. We also use improved Kato inequalities to deal with gradient terms. Integrating over the manifold would lead to a contradiction. For Theorem \ref{definitetheorem} and Theorem \ref{selfdualtheorem}, we apply that blueprint for harmonic two forms and harmonic Weyl tensor, respectively. It is also noted that the dimension (induced decomposition due to the Hodge star operator) comes into play in an essential way. 

Theorem \ref{Epositiveintersectionthm} has a slightly different flavor. The proof is based on an observation from \cite{caotran4}: making use of elliptic equations, which arise from Ricci flow computation, an upper bound would imply a lower bound. The rest follows from inequalities involving the Euler characteristic and Hirzebruch signature in a manner similar to \cite{gl99}. 

The organization of the paper is as follows. The next section collects preliminaries discussing the curvature decomposition and the relation between the geometry and topology of a closed four-manifold. We also describe some examples and list out various useful estimates. Section \ref{harmonic2f} carries out the proof for Theorem \ref{definitetheorem} and Corollary \ref{S2timeS2}. Then, in Section \ref{harmonicweyl}, we study a metric with harmonic Weyl curvature and prove Theorem \ref{selfdualtheorem} and Corollary \ref{classselfdual}. Finally, Section \ref{einsteinsection} collects the proof of Theorem \ref{Epositiveintersectionthm} and Corollary \ref{Eupperbound}.

{\bf Acknowledgment.} The authors would like to thank Ernani Ribeiro Jr. and Linlin Sun for their valuable comments and for pointing out reference \cite{cs18, DR18, DRF18, Ribeiro16} to us.  Cao's research was partially
supported by a grant from the Simons Foundation (\#280161 and \#585201). Part of this work was done while the second author was visiting the Vietnam Institute for Advanced Study in Mathematics (VIASM). He would like to thank VIASM for financial support and hospitality.  

\section{Preliminaries}
\label{prem}
In this section, we recall fundamental results regarding the geometry and topology of a four-dimensional manifold. Throughout, let $(M,\ g)$ denote an oriented smooth Riemannian manifold and its metric.  
\subsection{\textbf{Curvature decomposition}} 
The geometry of $(M, g)$ is determined by its Riemannian curvature $\RR$ which, due to its symmetry, can be considered as an adjoint operator on the vector bundle of 2-forms, $\wedge^2TM$. The algebraic structure of that vector bundle induces a decomposition of R into orthogonal components: 
\begin{equation}
\label{curvdec} 
\RR=\WW+\frac{\SS g\circ g}{2n(n-1)}+\frac{(\Rc-\frac{\SS}{n}\text{Id})\circ g}{n-2}.
\end{equation}
Here, $\Rc, \SS, \WW$ denote the Ricci curvature, scalar curvature and Weyl curvature, respectively. Also, the Kulkarni-Nomizu product $\circ$ is defined for symmetric 2-tensors $A, B$ as, for vector fields $x, y, z, w$,
\[ (A\circ B) (x, y, z, w):= A(x,z)B(y,w)+B(x,z)A(y,w)-A(x,w)B(y,z)-B(x,w)A(y,z).
\]

In dimension four, the Hodge star operator induces a natural decomposition of the vector bundle of 2-forms,
$\wedge^2 TM$ 
\begin{equation*}
\wedge^2 TM =\wedge^+ M \oplus \wedge^- M.
\end{equation*}
Here $\wedge^{\pm} M$ are the eigenspaces of eigenvalues $\pm 1$, respectively.
Elements of $\wedge^+ M$ and $\wedge^- M$ are called self-dual and
anti-self-dual 2-forms.

Furthermore, since the curvature can be considered as an operator on the space of 2-forms $\RR: \wedge^2 TM
\rightarrow \wedge^2 TM$, the Hodge star induces the following decomposition :
\begin{equation*}
\RR =\left( \begin{array}{cc}
\frac{\SS}{12}\id+\WW^+ & \frac{1}{2}(\Rc-\frac{\SS}{4}\id)\circ g \\
\frac{1}{2}(\Rc-\frac{\SS}{4}\id)\circ g & \frac{\SS}{12}\id+\WW^-
\end{array} \right).
\end{equation*}
Here, the self-dual and anti-self-dual Weyl curvature $\WW^{\pm}$ are the restriction of the Weyl curvature $\WW$ to self-dual and anti-self-dual 2-forms $\wedge^{\pm}M$, respectively. 

In addition, as $\WW$ is traceless and satisfies the first Bianchi identity, there is a normal form discovered by M. Berger \cite{berger61} (see also \cite{st69}).

\begin{prop}
\label{berger}
Let $(M,\ g)$ be a four-manifold. At each point $p\in M$, there exists an orthonormal basis $\{e_i\}_{1\leq i\leq
4}$ of $T_p M$, such that relative to the corresponding basis
$\{e_i\wedge e_j\}_{1\leq i<j\leq 4}$ of $\wedge^2 T_pM$,
$\WW$ takes the form
\begin{equation}
\label{abba}
\WW=\left( \begin{array}{cc}
A & B\\
B & A
\end{array}\right),
\end{equation}
where $A=\diag\{a_1,\ a_2,\ a_3\}$, $B=\diag \{b_1,\ b_2,\ b_3\}$.
Moreover, we have the followings:

\begin{enumerate}
\item  $a_1=\WW(e_1, e_2, e_1, e_2)=\WW(e_3, e_4, e_3, e_4)=\min_{|a|=|b|=1,~ a\perp b}\WW(a,b,a,b)$,
\item $a_3=\WW(e_1, e_4, e_1, e_4)=\WW(e_1, e_4, e_1, e_4)=\max_{|a|=|b|=1,~ a\perp b}\WW(a,b,a,b)$.

\item $a_2=\WW(e_1, e_3, e_1, e_3)=\WW(e_2, e_4, e_2, e_4)$, 

\item  $b_1=\WW_{1234},\ b_2=\WW_{1342},\ b_3=\WW_{1423}$,
\item $a_1+a_2+a_3=b_1+b_2+b_3=0$,

\item  $|b_2-b_1|\leq a_2-a_1,\ |b_3-b_1|\leq a_3-a_1,\ |b_3-b_2|\leq
a_3-a_2$.
\end{enumerate}
\end{prop}
Then, one can construct orthonormal bases for $\wedge^{\pm} M$ by, for $e_{ij}:=e_i\wedge e_j$,
\begin{align*}
\mathbb{B}^+ &=\frac{1}{\sqrt{2}}(e_{12}+e_{34}, e_{13}-e_{24}, e_{14}+e_{23}),\\
\mathbb{B}^- &=\frac{1}{\sqrt{2}}(e_{12}-e_{34}, e_{13}+e_{24}, e_{14}-e_{23}).
\end{align*}
As a result, eigenvalues of $\WW^{\pm}$ are ordered,
\begin{equation}
\label{eigenvalue}
\begin{cases}
\lambda_1^+=a_1+b_1 \leq \lambda_2=a_2+b_2 \leq \lambda_3^+=a_3+b_3, &\\
\lambda_1^-=a_1-b_1 \leq \lambda_2^-=a_2-b_2 \leq \lambda_3^-=a_3-b_3. &
\end{cases}
\end{equation}

The biorthogonal (sectional) curvature of a plane $P\in T_p(M)$ is defined as
\begin{equation*}
K^{\perp}(P)=\frac{K(P)+K(P^{\perp})}{2},
\end{equation*}
where $P^\perp$ is the orthogonal plane to $P$. If $P$ and $P^\perp$ are spanned by orthonomal bases $\{e_1, e_2\}$ and $\{e_3, e_4\}$ then
\begin{align}
K^\perp(P)&=\frac{1}{2}(\RR_{1212}+\RR_{3434})\nonumber\\
\label{orthogonalsec}
&=\WW_{1212}+\frac{\SS}{12}. 
\end{align}
Here we use equation $(\ref{curvdec})$ to simplify the computation. It is noted that, for $P_1, P_2, P_3$ planes spanned by $\{e_1, e_2\}, \{e_1, e_2\}, \{e_1, e_2\}$ respectively,  
\begin{equation}
\label{scalarandKperp}
\frac{\SS}{4}=\KK^\perp(P_1)+\KK^\perp(P_2)+\KK^\perp(P_3)
\end{equation}
\subsection{\textbf{Topology and Harmonic Forms}}
We will describe the topology of a closed connected four-manifold. Generally, for any manifold $M$ of dimension $n$, the $k$-th Betti number $b_k(M)$, intuitively the number of $k$-dimensional holes, is the rank of the $k$-th homology group. If $M$ is closed and oriented, by Poincare's duality,
\[ b_k(M)= b_{n-k}(M).\]
Thus, if $M$ is connected 
\[b_0(M)=b_n(M)=1.\]
Other topological invariants, subsequently, can be expressed in terms of these numbers. Notably, the Euler characteristic is given by
\[ \chi(M)=\sum_{i=0}^{\infty}(-1)^i b_i(M).\] 

Next, we restrict to dimension four. The Euler characteristic for a closed, connected, and oriented manifold $M$ is 
\[\chi(M)=2-2b_1+b_2.\]
\begin{remark}
When $M$ has a finite fundamental group, then $b_1=0$. 
\end{remark}
When $M$ is equipped with a Riemannian metric $g$, the Gauss-Bonnet-Chern formula states that 
\begin{equation}
\label{Euler}
8\pi^2 \chi(M) = \int_{M}(|\WW|^2-\frac{1}{2}|\Rc-\frac{\SS}{4}\id|^2+\frac{S^2}{24}) dv
\end{equation}

Furthermore, by De Rham's theorem, $b_k$ is also the dimension of the space of harmonic $k$-forms. The decomposition induced by the Hodge star operator translates into
\[b_2(M)=b_+(M) + b_-(M).\]
Here, $b_{\pm}$ are the dimension of the space of harmonic self-dual and anti-self-dual 2-forms, respectively. What is more, the difference between $b_{\pm}$ is also a topological invariant, called the signature. Analogous to the Euler characteristic, Hirzebruch also found a formula for the signature using curvature terms (cf. \cite{Besse} for more details)
\begin{equation}
\label{signature}
b_{+}(M)-b_-(M):=\tau(M)=\frac{1}{12\pi^2}\int_{M}(|\WW^{+}|^2-|\WW^{-}|^2)dv. 
\end{equation} 
\begin{definition}
A closed, oriented, connected, smooth manifold in dimension four is said to be definite if $b_{+}b_{-}=0$. It is said to have positive intersection form if it is definite and $b_2>0$.  
\end{definition}

Next, we recall a Bochner formula for harmonic 2-forms. In general, interchanging the order of derivative gives rise to curvature terms. Specifically, for any two-forms $\omega$, 
\[\Delta |\omega|^2=2\left\langle{\Delta \omega, w}\right\rangle+2|\nabla \omega|^2+2 \RR_2(\omega, \omega).
\]
Here, using equation (\ref{curvdec}),
\[
\RR_2= \Rc \circ g -2\RR=\frac{\SS}{4}g\circ g -2\WW-\frac{\SS}{12} g\circ g=\frac{\SS}{6}g\circ g -2\WW.
\]
It is noted that $\RR_2$ also goes by the name Weitzenbock operator. And a manifold is said to have non-negative isotropic curvature if $\RR_2\geq 0$. Since $\RR_2=\frac{\SS}{6}g\circ g- \WW$, with respect to $\wedge^2 =\wedge^+ \oplus \wedge^-$, 
\begin{equation*}
\RR_2=\left( \begin{array}{cc}
\frac{\SS}{3}\id-2\WW^+ & 0 \\
0 & \frac{\SS}{3}\id-\WW^-
\end{array} \right)
\end{equation*}
Thus, if $\lambda^{\pm}_1\leq \lambda^{\pm}_2\leq \lambda^{\pm}_3$ are eigenvalues of $\WW^{\pm}$, then,
\begin{equation}
\label{estimateR2} (\frac{\SS}{3}-2\lambda^{\pm}_1) |\omega_{\pm}|^2\geq \RR_2(\omega_{\pm}, \omega_{\pm}) \geq (\frac{\SS}{3}-2\lambda^{\pm}_3)|\omega_{\pm}|^2.
\end{equation}
If $\omega$ is harmonic then we have,
\begin{equation}
\label{bochnerhar}\Delta |\omega|^2=2|\nabla \omega|^2+2 \RR_2(\omega, \omega).
\end{equation} 
Furthermore, there is an improved Kato's inequality discovered by W. Seaman \cite{seaman93} for harmonic 2-forms:
\begin{equation}
\label{iKatohar}|\nabla \omega|^2 \geq \frac{3}{2}\nabla |\omega||^2.
\end{equation}

\subsection{\textbf{Examples}}
Here we describe the geometry and topology of some well-known simply connected 4-manifolds. As a consequence, they all have $b_1=0$. So all topological invariants discussed above are totally determine by the Euler characteristic and the Hirzebruch signature. 

First, the sphere has the following topological invariants: \[\chi=2 \text{~and~} \tau=0.\]
The curvature of the round metric $g_0$ on $\mathbb{S}^4$ is:
\begin{equation}
\RR=\left( \begin{array}{cc}
\frac{\SS}{12}\id &  \\
   & \frac{\SS}{12}\id
     \end{array} \right)
\end{equation}
Then, the real projective space $(\mathbb{RP}^4, g_0)$ is the quotient of $(\mathbb{S}^4,g_0)$ by the antipodal identification.

The complex projective space $\mathbb{CP}^2$ has the following topological invariants: 
\[\chi=3 \text{~and~} \tau=1.\]
With some orientation, the curvature of the Fubini-Study metric $g_{FS}$ on $\mathbb{CP}^2$ is:
\begin{equation}
\RR=\left( \begin{array}{cc}
\diag \{0,\ 0, \ \frac{\SS}{4}\}&  \\
   & \frac{\SS}{12}\id
     \end{array} \right).
\end{equation}
The self-dual part of Weyl tensor $\WW^+=\diag\{-\frac{\SS}{12}, \ -\frac{\SS}{12}, \ \frac{\SS}{6}\}$ and anti-self-dual part $\WW^{-}=0$.
 
The product of 2 spheres $\mathbb{S}^2\times \mathbb{S}^2$ has the following topological invariants: 
\[\chi=4 \text{~and~} \tau=0.\]
The curvature of the product metric on $\mathbb{S}^2\times \mathbb{S}^2$ is 
\begin{equation}
\RR=\left( \begin{array}{cc}
A &  0\\
0  & A
     \end{array} \right)
\end{equation} 
for $A=\diag\{0,\ 0,\ \frac{\SS}{4}\}$. The self-dual part and anti-self-dual part of the Weyl tensor are
$\WW^{\pm}=\diag\{-\frac{\SS}{12},\ -\frac{\SS}{12},\ \frac{\SS}{6}\}$. \\

\subsection{\textbf{Basic Estimates}}

We first relate eigenvalues of $\WW^{\pm}$ with the scalar curvature and the biorthogonal curvature. 
\begin{lemma}
	\label{estimatelambda}
	Let $(M,g)$ be a 4-manifold and $\delta(p)=\max_{P\in T_p(M)}\KK^\perp(P)$. Then, 
	\begin{align}
	\frac{2\SS}{3}-2\lambda_3^+-2\lambda_3^- &\geq \SS-4\delta,\\
	|\lambda_3^+-\lambda_3^-| &\leq 2(\delta-\frac{\SS}{12}),\\
	\frac{\SS}{3}-2\lambda_3^{\pm} &\geq \frac{2\SS}{3}-4\delta.
	\end{align}
	Furthermore, equality happens in the last two if and only if $|\WW^-||\WW^+|=0$.
\end{lemma}
\begin{proof}
	Using Berger's normal form as described in Prop. \ref{berger}, we have 
	\[ 2\lambda_3^+ + 2\lambda_3^-=4\WW_{1414}=4\WW_{2323}.\]
	On the other hand, by equation (\ref{orthogonalsec})
	\[\WW_{1414}=K^{\perp}(P)-\frac{\SS}{12},\]
	where $P$ is the tangent plane spanned by $\{e_1, e_4\}$. 
	Thus, 
	\begin{align*}
	\frac{2\SS}{3}-2\lambda_3^+-2\lambda_3^-&= \SS-4K^\perp(P),\\
	&\geq \SS-4\delta.
	\end{align*}
	Next, we observe
	\[
	|\lambda_3^+-\lambda_3^-| = 2|\WW_{1423}|.\]
	Furthermore, 
	\begin{align*}
	0\leq \lambda_3^+\lambda_3^- &=\frac{1}{4}(\WW_{1414}+\WW_{2323})^2-|\WW_{1423}|^2\\
	&=\frac{1}{4}(2\KK^\perp(P)-\frac{\SS}{6})^2-|\WW_{1423}|^2\\
	&\leq \frac{1}{4}(2\delta-\frac{\SS}{6})^2-|\WW_{1423}|^2.
	\end{align*}
	Therefore, the second inequality follows. Also, equality happens if and only if $0=\lambda_3^+\lambda_3^-$. Since each is the largest eigenvalue, $|\WW^-||\WW^+|=0$.

	Finally, we compute,
	\begin{align*}
	\frac{\SS}{3}-2\lambda_3^+ &=\frac{\SS}{3}-(2\KK^\perp(P)-\frac{\SS}{6}+2\WW_{1423})\\
	&=\frac{\SS}{2}-2\KK^\perp(P)-2\WW_{1423}\\
	&\geq \frac{\SS}{2}-2\KK^\perp(P)-2|\WW_{1423}|\\
	&\geq \frac{\SS}{2}-2\delta-2(\delta-\frac{\SS}{12})=\frac{2\SS}{3}-4\delta.
	\end{align*}
\end{proof}

The next result estimates $\det{\WW^\pm}$ and will be used in Section \ref{harmonicweyl}. 
\begin{lemma} 
	\label{estimatedet}
	Let $(M,g)$ be a 4-manifold then
	\begin{equation}
	36 \det{\WW^\pm}\leq 6\lambda_3^{\pm}|\WW^\pm|^2.
	\end{equation}
\end{lemma}
\begin{proof}
	It suffices to prove it for $\WW^+$ as the other case is similar. Using the normal form as described in Prop \ref{berger}, we compute
	\begin{align*}
	6\lambda_3^{+}|\WW^+|^2-36 \det{\WW^+}&=12\lambda^+_3\Big((\lambda^+_3)^2+(\lambda^+_2)^2+\lambda_3^+\lambda_2^+\Big)-36\lambda^+_1\lambda^+_2\lambda^+_3\\
	&=12(\lambda^+_3)^3-48\lambda^+_1\lambda^+_2\lambda^+_3\\
	&=12\lambda^+_3(\lambda^+_1-\lambda^+_2)^2\\
	&\geq 0.
	\end{align*}
	Here, we repeatedly utilize the fact that $\lambda^+_1+\lambda^+_2+\lambda^+_3=0$. 
\end{proof}

Finally, we have the following lemma which relates all curvature assumptions discussed in the Introduction. 
\begin{lemma}
	\label{relatedconditions}
	Let $(M,g)$ be a closed four-dimensional manifold with positive scalar curvature and let $\lambda_1$ be its first eigenvalue of the Laplacian on functions. Suppose one of the following conditions holds:
	\begin{enumerate}
		\item $\KK^\perp\geq \frac{\SS^2}{24(\SS+3\lambda)};$
		\item $\KK^\perp_{min}\geq \frac{\SS}{2(2\SS+9\lambda_1)} \KK^\perp_{max}$;
		\item There exists $k>0$, such that $\Rc\geq k$ and $\KK^\perp \leq \frac{5}{6}k$; 
	\end{enumerate}
	Then \[\KK^\perp\leq \frac{\SS(2\SS+9\lambda_1)}{12(\SS+3\lambda_1)}.\]
\end{lemma}
\begin{proof}
For an orthonormal basis $\{e_1, e_2, e_3, e_4\}$, let $P_1, P_2, P_3$ be planes spanned by $\{e_1, e_2\}$, $\{e_1, e_3\}$, $\{e_1, e_4\}$ respectively. Recall that 
\[\frac{\SS}{4}=\KK^\perp(P_1)+\KK^\perp(P_2)+\KK^\perp(P_3).\]
So if $\KK^\perp\geq \frac{\SS^2}{24(\SS+3\lambda)}$ then 
\begin{align*}
\KK^\perp(P_3) &\leq \frac{\SS}{4}-2 \frac{\SS^2}{24(\SS+3\lambda)} \\
&\leq \frac{\SS(2\SS+9\lambda_1)}{12(\SS+3\lambda_1)}.
\end{align*}
Thus $(1)$ implies the desired conclusion. Similarly, if $(2)$ holds then
\begin{align*}
\frac{\SS}{4} &\geq \KK^\perp(P_3)+2\frac{\SS}{2(2\SS+9\lambda_1)} \KK^\perp(P_3),\\
\KK^\perp(P_3) &\leq \frac{\SS(2\SS+9\lambda_1)}{12(\SS+3\lambda_1)}.
\end{align*}
Finally, if $\Rc\geq k>0$ then $\SS\geq 4k$ and, by Lichnerowicz's theorem \cite{lich58}, $\lambda\geq \frac{4}{3}k$. Furthermore, it is noted that the function $f(x,y)=\frac{x(2x+9y)}{12(x+3y)}$, defined on $x,y>0$, is increasing on both $x$ and $y$. Thus, 
\begin{align*}
\frac{\SS(2\SS+9\lambda_1)}{12(\SS+3\lambda_1)} &\geq \frac{4k(8k+12k)}{12(4k+4k)}\\
&\geq \frac{5}{6}k. 
\end{align*}  
\end{proof}
\section{Harmonic Two-Forms}
\label{harmonic2f}
In this section, we study harmonic 2-forms and prove Theorem \ref{definitetheorem} and Corollary \ref{S2timeS2}. First, we deduce a general integral inequality relating the norms of a harmonic self-dual and anto-self-dual two-forms with the first eigenvalue of the Laplacian.

\begin{prop}
\label{integralestimate}
Let $(M,g)$ be a closed four-dimensional manifold and $\lambda_1$ be its first eigenvalue of the Laplacian on functions. Suppose there are harmonic self-dual and anti-self-dual 2 forms $\omega_{\pm}$ such that 
\[ \int_{M} |\omega_{+}|^{\alpha}= t \int_{M} |\omega_{-}|^{\alpha}.\]
Then we have the following inequality,
\begin{align*}
0\geq \int_M & 2\alpha\Big(|\omega_+|^{2(\alpha-1)}\RR_2(\omega_+, \omega_+)+t^2|\omega_-|^{2(\alpha-1)}\RR_2(\omega_-, \omega_-)\Big)\\
&+\frac{\lambda_1(4\alpha-1)}{2\alpha}\Big(|\omega_+|^\alpha-t|\omega_-|^\alpha\Big)^2.
\end{align*}
\end{prop}
\begin{proof}
To make the calculation clean, we'll assume that $|\omega_{\pm}|>0$ (if $|\omega_{\pm}|=0$ at some points, replace $|\omega_{\pm}|^{2\alpha}$ by $|\omega_{\pm}|^{2\alpha}+\epsilon$ and let $\epsilon\rightarrow 0$; see \cite{caotran4} for details). Using (\ref{bochnerhar}), we compute, for any harmonic 2-forms $\omega$, 
\begin{align*}
\Delta |\omega|^{2\alpha} &= \alpha(\alpha-1)|\omega|^{2(\alpha-2)}|\nabla |\omega|^2|^2+\alpha|\omega|^{2(\alpha-1)}\Delta |\omega|^2,\\
&= |\omega|^{2(\alpha-1)}2\alpha\Big((|\nabla \omega|^2+\RR_2(\omega, \omega))+2(\alpha-1)|\nabla |\omega||^2\Big).
\end{align*}
Thus, for harmonic self-dual and anti-self-dual forms $\omega_{\pm}$,
\begin{align*}
\Delta(|\omega_{+}|^{2\alpha}+ t^2\omega_{-}|^{2\alpha}) =&
|\omega_+|^{2(\alpha-1)}2\alpha\Big((|\nabla \omega_+|^2+\RR_2(\omega_+, \omega_+))+2(\alpha-1)|\nabla |\omega_+||^2\Big)\\
+t^2& |\omega_-|^{2(\alpha-1)}2\alpha\Big((|\nabla \omega_-|^2+\RR_2(\omega_-, \omega_-))+2(\alpha-1)|\nabla |\omega_-||^2\Big).
\end{align*}
Using the improved Kato's inequality (\ref{iKatohar}) and integrating over the manifold yield,
\begin{align*}
0\geq \int_M &|\omega_+|^{2(\alpha-1)}\alpha\Big( (4\alpha-1)|\nabla |\omega_+||^2+2 \RR_2(\omega_+, \omega_+)\Big)\\
+t^2 & |\omega_-|^{2(\alpha-1)}\alpha\Big( (4\alpha-1) |\nabla |\omega_-||^2+2 \RR_2(\omega_-, \omega_-)\Big),\\
0\geq \int_M & 2\alpha\Big(|\omega_+|^{2(\alpha-1)}\RR_2(\omega_+, \omega_+)+t^2|\omega_-|^{2(\alpha-1)}\RR_2(\omega_-, \omega_-)\Big)\\
&+\frac{4\alpha-1}{\alpha}\Big(|\nabla |\omega_+|^\alpha|^2+t^2|\nabla |\omega_-|^\alpha|^2\Big).
\end{align*}
By the variation characterization of $\lambda_1$, 
\begin{align*}
0\geq \int_M & 2\alpha\Big(|\omega_+|^{2(\alpha-1)}\RR_2(\omega_+, \omega_+)+t^2|\omega_-|^{2(\alpha-1)}\RR_2(\omega_-, \omega_-)\Big)\\
&+\frac{4\alpha-1}{2\alpha}\Big(\nabla (|\omega_+|^\alpha-t|\omega_-|^\alpha)\Big)^2,\\
\geq \int_M & 2\alpha\Big(|\omega_+|^{2(\alpha-1)}\RR_2(\omega_+, \omega_+)+t^2|\omega_-|^{2(\alpha-1)}\RR_2(\omega_-, \omega_-)\Big)\\
&+\frac{\lambda_1(4\alpha-1)}{2\alpha}\Big(|\omega_+|^\alpha-t|\omega_-|^\alpha\Big)^2.
\end{align*}

By rearranging the terms we arrive at

\begin{align*} 
0\geq \int_M t^2& |\omega_-|^{2(\alpha-1)}\Big(2\alpha\RR_2(\omega_-, \omega_-)+ \frac{\lambda_1(4\alpha-1)}{2\alpha}|\omega_-|^{2}\Big)\\
+&|\omega_+|^{2(\alpha-1)}\Big(2\alpha\RR_2(\omega_+, \omega_+)+ \frac{\lambda_1(4\alpha-1)}{2\alpha}|\omega_+|^{2}\Big)\\
-&2t \frac{\lambda_1(4\alpha-1)}{2\alpha}|\omega_-|^{\alpha} |\omega_+|^{\alpha}.
\end{align*}

\end{proof}

We are now ready to prove the main theorem of this section. 
\begin{theorem}
\label{definiteupper}
Let $(M,g)$ be a closed four-dimensional manifold with positive scalar curvature and let $\lambda_1>0$ be its first eigenvalue of the Laplacian on functions. Suppose that \[\KK^\perp\leq \frac{\SS(2\SS+9\lambda_1)}{12(\SS+3\lambda_1)}.\] Then $M$ has definite intersection form.
\end{theorem}
\begin{proof}
We prove by contradiction. Suppose the statement is false then there are non-trivial self-dual harmonic and anti-self-dual harmonic 2-forms $\omega_{\pm}$.  The assumptions of Prop \ref{integralestimate} are satisfied and we have,
\begin{align*} 
0\geq \int_M t^2& |\omega_-|^{2(\alpha-1)}\Big(2\alpha\RR_2(\omega_-, \omega_-)+ \frac{\lambda_1(4\alpha-1)}{2\alpha}|\omega_-|^{2}\Big)\\
+&|\omega_+|^{2(\alpha-1)}\Big(2\alpha\RR_2(\omega_+, \omega_+)+ \frac{\lambda_1(4\alpha-1)}{2\alpha}|\omega_+|^{2}\Big)\\
-&2t \frac{\lambda_1(4\alpha-1)}{2\alpha}|\omega_-|^{\alpha} |\omega_+|^{\alpha}.
\end{align*}

We can choose $\alpha=\frac{1}{2}$ to maximize $\frac{4\alpha-1}{\alpha^2}$. Thus, 
\begin{align*} 
0\geq \int_M t^2& |\omega_-|^{-1}\Big(\RR_2(\omega_-, \omega_-)+ \lambda_1|\omega_-|^{2}\Big)\\
+&|\omega_+|^{-1}\Big(\RR_2(\omega_+, \omega_+)+ \lambda_1|\omega_+|^{2}\Big)\\
-&2t \lambda_1|\omega_-|^{1/2} |\omega_+|^{1/2}.
\end{align*}
The integrand is a quadratic polynomial on $t$. Using (\ref{estimateR2}) and Lemma \ref{estimatelambda}, the leading term is at least 
\begin{align*}
|\omega_-|^{-1}\Big(\RR_2(\omega_-, \omega_-)+ \lambda_1|\omega_-|^{2}\Big) \geq & |\omega_-|(\frac{\SS}{3}-2\lambda_3^-+\lambda_1),\\
\geq & |\omega_-|(\frac{2\SS}{3}-4\frac{\SS(2\SS+9\lambda_1)}{12(\SS+3\lambda_1)}+\lambda_1),\\
\geq & |\omega_-|\lambda_1(1-\frac{\SS}{\SS+3\lambda_1})>0. 
\end{align*}
Now we compute its discriminant
\begin{align*}
D &= |\omega_-||\omega_+|\lambda_1^2-|\omega_-|^{-1}|\omega_+|^{-1}\Big(\RR_2(\omega_-, \omega_-)+ \lambda_1|\omega_-|^{2}\Big)\Big(\RR_2(\omega_+, \omega_+)+ \lambda_1|\omega_+|^{2}\Big).
\end{align*}
Since each term $\RR_2(\omega_{\pm}, \omega_{\pm})+ \lambda_1|\omega_{\pm}|^{2}> 0$, 
\begin{align*}
D &\leq |\omega_-||\omega_+|\Big(\lambda_1^2-(\frac{\SS}{3}-2\lambda_3^-+\lambda_1)(\frac{\SS}{3}-2\lambda_3^++\lambda_1)\Big)\\
& \leq |\omega_-||\omega_+|\Big(-(\frac{\SS}{3}-2\lambda_3^-)(\frac{\SS}{3}-2\lambda_3^+)-\lambda_1(\frac{2\SS}{3}-2\lambda_3^--2\lambda_3^+)\Big)\\
& \leq |\omega_-||\omega_+|\Big(|\lambda_3^+-\lambda_3^-|^2-(\frac{\SS}{3}-\lambda_3^--\lambda_3^+)^2-\lambda_1(\frac{2\SS}{3}-2\lambda_3^--2\lambda_3^+)\Big).
\end{align*}
Applying Lemma \ref{estimatelambda} again yields
\begin{align*} D &\leq |\omega_-||\omega_+|\Big(4(\frac{\SS(2\SS+9\lambda_1)}{12(\SS+3\lambda_1)}-\frac{\SS}{12})^2-\frac{1}{4}(\SS-4\frac{\SS(2\SS+9\lambda_1)}{12(\SS+3\lambda_1)})^2-\lambda_1(\SS-4\frac{\SS(2\SS+9\lambda_1)}{12(\SS+3\lambda_1)}))\Big)\\
&\leq |\omega_-||\omega_+|\Big(\frac{\SS^2}{36(\SS+3\lambda_1)^2}((\SS+6\lambda_1)^2-\SS^2)-\lambda_1\frac{\SS^2}{3(\SS+3\lambda_1)}\Big)\\
&\leq 0. 
\end{align*}
Thus, the integrand must be vanishing at each point and all inequalities above assume equality. In particular, by Lemma \ref{estimatelambda}, at each point
\[|\WW^-||\WW^+|=0.\]
If $|\WW^-|=0$, the integrand becomes
\begin{align*}
& t^2 |\omega_-|(\frac{\SS}{3}+\lambda_1)+|\omega_+|(\frac{\SS}{3}-2\lambda_3+\lambda_1)-2t \lambda_1|\omega_-|^{1/2} |\omega_+|^{1/2}\\
=&t^2 |\omega_-|(\frac{\SS}{3}+\lambda_1)+|\omega_+|(\frac{2\SS}{3}-4\frac{\SS(2\SS+9\lambda_1)}{12(\SS+3\lambda_1)}+\lambda_1)-2t\lambda_1|\omega_-|^{1/2} |\omega_+|^{1/2}
\\=&\Big(t|\omega_-|^{1/2}\sqrt{\frac{\SS+3\lambda_1}{3}}-\lambda_1\sqrt{\frac{3}{\SS+3\lambda_1}}|\omega_+|^{1/2}\Big)^2.
\end{align*}
Similarly, if $|\WW^+|=0$ then the integrand is
\[\Big(|\omega_+|^{1/2}\sqrt{\frac{\SS+3\lambda_1}{3}}-\lambda_1\sqrt{\frac{3}{\SS+3\lambda_1}}t|\omega_-|^{1/2}\Big)^2.\]
Thus, at each point, one of the following must hold, 
\begin{align}
\label{eq1}
|\omega_+|^{1/2} &= \frac{\SS+3\lambda_1}{3\lambda_1}t|\omega_-|^{1/2},\\
\label{eq2}
|\omega_+|^{1/2} &= \frac{3\lambda_1}{\SS+3\lambda_1}t|\omega_-|^{1/2}.
\end{align}
Next, the integral estimate assumes equality only if, by the proof of Prop. \ref{integralestimate},  
\begin{align}
\label{eq3}
\int_{M} |\omega_{+}|^{1/2}&= t \int_{M} |\omega_{-}|^{1/2};\\
\label{eq4}
0 &=\nabla(|\omega_{+}|^{1/2}+t |\omega_{-}|^{1/2}).
\end{align}
Let $\Omega_1, \Omega_2$ be the sets of points where (\ref{eq1}) and (\ref{eq2}) hold, respectively. If both are non-empty then they share a boundary on which $|\omega_{\pm}|=0$ because $\frac{\SS+3\lambda_1}{3\lambda_1}>1$. However, (\ref{eq4}) then implies that $|\omega_{\pm}|=0$ everywhere, a contradiction. So either $\Omega_1$ or $\Omega_2$ is empty and the other is the whole manifold. In that case, comparing with (\ref{eq3}) also leads to a contradiction.   

\end{proof}
Now we are ready to prove Theorem \ref{definitetheorem} and Corollary \ref{S2timeS2}.
\begin{proof}(\textbf{of Theorem \ref{definitetheorem}}) The result follows from Lemma \ref{relatedconditions} and Theorem \ref{definiteupper}.
\end{proof}
\begin{proof}(\textbf{of Corollary \ref{S2timeS2}}) From Section \ref{prem}, we know that $\mathbb{S}^2\times \mathbb{S}^2$ is indefinite as $b_+=b_-=1$. Thus, the result follows from Theorem \ref{definitetheorem}.

\end{proof}
\section{Harmonic Weyl Tensor}
\label{harmonicweyl}
In this section, we study a 4-manifold with harmonic Weyl curvature. Such a Riemannian manifold is characterized by the equation 
\[\delta \WW=0,\]
where $\delta$ is the divergent operator. Notably, this condition is a generalization of the Einstein equation. Indeed, an Einstein structure has constant Ricci curvature. Then, the weaker condition of having parallel Ricci tensor is equivalent to harmonic curvature which, in turn, means harmonic Weyl curvature and constant scalar curvature.  

For a manifold in dimension four, the decomposition induced by the Hodge star operator leads to 
\[\delta \WW^{\pm}=0.\]  
A. Derdzinski \cite{Derd83} observed the following Bochner formula. 
\begin{equation}
\label{BochnerW}
\Delta|\WW^{\pm}|^2 =2|\nabla\WW^\pm|^2+\SS|\WW^{\pm}|^2-36\det\WW^{\pm}.\end{equation}
Furthermore, there is an improved Kato's inequality observed by Gursky-LeBrun \cite{gl99} and Yang \cite{yangdg00} (shown to be optimal by \cite{branson00, CGH00weight}):
\begin{equation}
\label{improvedKW}
|\nabla\WW^{\pm}|^2 \geq \frac{5}{3}|\nabla|\WW^\pm||^2.
\end{equation}
Using equations (\ref{BochnerW}), (\ref{improvedKW}) and the same procedure as in Prop. \ref{integralestimate} yields the following. 
\begin{prop}

\label{Wintegralestimate}
Let $(M,g)$ be a closed four-dimensional manifold with harmonic Weyl curvature and $\lambda_1$ be its first eigenvalue of the Laplacian on functions. Suppose there exists $t>0$ such that 
\[ \int_{M} |\WW^{+}|^{\alpha}= t \int_{M} |\WW^{-}|^{\alpha},\]
then we have the following inequality,
\begin{align*}
0\geq \int_M & \alpha \Big(|\WW^+|^{2(\alpha-1)}(\SS|\WW^+|^2-36\det{\WW^+})+t^2|\WW^-|^{2(\alpha-1)}(\SS|\WW^-|^2-36\det{\WW^-})\Big)\\
&+\frac{\lambda_1(6\alpha-1)}{3\alpha}\Big(|\WW^+|^\alpha-t|\WW^-|^\alpha\Big)^2.
\end{align*}
\end{prop}

We can now state our main theorem in this section as the following. 

\begin{theorem}
\label{selfdualupper}
	Let $(M,g)$ be a closed four-dimensional manifold with harmonic Weyl curvature and positive scalar curvature and  $\lambda_1>0$ be its first eigenvalue of the Laplacian on functions. Suppose that \[\KK^\perp\leq \frac{\SS(2\SS+9\lambda_1)}{12(\SS+3\lambda_1)},\] then $M$ is either self-dual or anti-self-dual.
\end{theorem}
\begin{proof}
	We prove by contradiction. Suppose the statement is false then there are some $t>0, \alpha>0 $ such that
	\[ \int_{M} |\WW^{+}|^{\alpha}= t \int_{M} |\WW^{-}|^{\alpha}.\]
	Prop. \ref{integralestimate} yields that 
	\begin{align*} 
	0\geq \int_M t^2& |\WW^-|^{2(\alpha-1)}\Big(\alpha(\SS|\WW^-|^2-36\det{\WW^-})+ \frac{\lambda_1(6\alpha-1)}{3\alpha}|\WW^-|^{2}\Big)\\
	+&|\WW^+|^{2(\alpha-1)}\Big(\alpha(\SS|\WW^+|^2-36\det{\WW^+})+ \frac{\lambda_1(6\alpha-1)}{3\alpha}|\WW^+|^{2}\Big)\\
	-&2t \frac{\lambda_1(6\alpha-1)}{3\alpha}|\WW^-|^{\alpha} |\WW^+|^{\alpha}.
	\end{align*}
	We now choose $\alpha=\frac{1}{3}$ to maximize $\frac{6\alpha-1}{\alpha^2}$. Thus, 
	\begin{align*} 
	0\geq \int_M t^2& |\WW^-|^{-4/3}\Big(\frac{1}{3}(\SS|\WW^-|^2-36\det{\WW^-})+ \lambda_1|\WW^-|^{2}\Big)\\
	+&|\WW^+|^{-4/3}\Big(\frac{1}{3}(\SS|\WW^+|^2-36\det{\WW^+})+ \lambda_1|\WW^+|^{2}\Big)\\
	-&2t \lambda_1|\WW^-|^{1/3} |\WW_+|^{1/3}.
	\end{align*}
	The integrand is a quadratic polynomial of $t$. Using (\ref{estimatedet}) and Lemma \ref{estimatelambda}, the leading term is at least 
	\begin{align*}
	|\WW^-|^{-4/3}\Big(\frac{1}{3}(\SS|\WW^-|^2-36\det{\WW^-})+ \lambda_1|\WW^-|^{2}\Big) \geq & |\WW^-|^{2/3}(\frac{\SS}{3}-2\lambda_3^-+\lambda_1),\\
	\geq & |\WW^-|^{2/3}(\frac{2\SS}{3}-4\frac{\SS(2\SS+9\lambda_1)}{12(\SS+3\lambda_1)}+\lambda_1),\\
	\geq & |\WW^-|^{2/3}\lambda_1(1-\frac{\SS}{\SS+3\lambda_1})>0. 
	\end{align*}
	Now we compute its discriminant
	\begin{align*}
	D &= |\WW^-|^{2/3}|\WW^+|^{2/3}\lambda_1^2\\
&-|\WW^-|^{-4/3}|\WW^+|^{-4/3}\Big(\frac{1}{3}(\SS|\WW^-|^2-36\det{\WW^-})+ \lambda_1|\WW^-|^{2}\Big)\Big(\frac{1}{3}(\SS|\WW^+|^2-36\det{\WW^+})+ \lambda_1|\WW^+|^{2}\Big).
	\end{align*}
	Since each term $\frac{1}{3}(\SS|\WW^\pm|^2-36\det{\WW^\pm})+ \lambda_1|\WW^\pm|^{2}\geq 0$, 
	\begin{align*}
	D &\leq |\WW^-|^{2/3}|\WW^+|^{2/3}\Big(\lambda_1^2-(\frac{\SS}{3}-2\lambda_3^-+\lambda_1)(\frac{\SS}{3}-2\lambda_3^++\lambda_1)\Big)\\
	& \leq |\WW^-|^{2/3}|\WW^+|^{2/3}\Big(-(\frac{\SS}{3}-2\lambda_3^-)(\frac{\SS}{3}-2\lambda_3^+)-\lambda_1(\frac{2\SS}{3}-2\lambda_3^--2\lambda_3^+)\Big)\\
	& \leq |\WW^-|^{2/3}|\WW^+|^{2/3}\Big(|\lambda_3^+-\lambda_3^-|^2-(\frac{\SS}{3}-\lambda_3^--\lambda_3^+)^2-\lambda_1(\frac{2\SS}{3}-2\lambda_3^--2\lambda_3^+)\Big).
	\end{align*}
	Applying Lemma \ref{estimatelambda} again yields
	\begin{align*} D &\leq |\omega_-||\omega_+|\Big(4(\frac{\SS(2\SS+9\lambda_1)}{12(\SS+3\lambda_1)}-\frac{\SS}{12})^2-\frac{1}{4}(\SS-4\frac{\SS(2\SS+9\lambda_1)}{12(\SS+3\lambda_1)})^2-\lambda_1(\SS-4\frac{\SS(2\SS+9\lambda_1)}{12(\SS+3\lambda_1)}))\Big),\\
	&\leq |\omega_-||\omega_+|\Big(\frac{\SS^2}{36(\SS+3\lambda_1)^2}((\SS+6\lambda_1)^2-\SS^2)-\lambda_1\frac{\SS^2}{3(\SS+3\lambda_1)}\Big)\\
	&\leq 0. 
	\end{align*}
	Thus, the integrand must be vanishing at each point and all inequalities above assume equality. In particular, by Lemma \ref{estimatelambda}, at each point
	\[|\WW^-||\WW^+|=0.\]
	If $|\WW^-|=0$ the integrand becomes \[|\WW^+|^{-4/3}\Big(\frac{1}{3}(\SS|\WW^+|^2-36\det{\WW^+}).\]
	Thus, it is vanishing if $|\WW^+|=0$. So $|\WW^-|=0=|\WW^+|$. A similar argument applies when $|\WW^+|=0$ to conclude that  $|\WW^-|=0=|\WW^+|$ everywhere, this is a contradiction. 
	\end{proof} 
	Theorem \ref{selfdualtheorem} and Corollary \ref{classselfdual} now follow immediately. 
	\begin{proof}(\textbf{of Theorem \ref{selfdualtheorem}}) The result follows from Lemma \ref{relatedconditions} and Theorem \ref{selfdualupper}.
	
	\end{proof}
\begin{proof}(\textbf{of Corollary \ref{classselfdual}}) Let define the following set 
\[ \Omega_1:=\{ p\in M; \Rc(p)\neq \frac{\SS}{4}g\}.\]
By a result of Derdzinski \cite[Corollary 1]{Derd88}, at a point (if any) $p\in \Omega_1$, $\WW^{\pm}$ have the same spectra, including multiplicities. That is $|\WW^+|=|\WW^-|=0$. 

If $\Omega_1$ is empty, then $(M, g)$ is an Einstein metric. By Theorem \ref{selfdualtheorem}, $(M,g)$ must be a self-dual or anti-self-dual Einstein manifold. Then it must be homothetically isometric to $\mathbb{CP}^2$ with its Fubini-Study metric or $\mathbb{S}^4$ with the round metric or its quotient by Hitchin's theorem \cite[Theorem 13.30]{Besse}. 

Otherwise, $\Omega_1$ is non-empty and contains an open set. In this set, $|\WW^+|=|\WW^-|=0$. The analyticity assumption then implies that $|\WW|\equiv 0$ everywhere. That is, $(M,g)$ is locally conformally flat.   	
	\end{proof}	
\section{Einstein Structures}
\label{einsteinsection}

In this section, we investigate an Einstein manifold with positive scalar curvature. A Riemannian manifold $(M,g)$ is called Einstein if it satisfies 
\begin{equation}
\label{einstein}
\Rc=\lambda g,
\end{equation}
where $\lambda$ is a constant. By rescaling if necessary, we can assume that 
\[\Rc=g.\]
Generally, by Myer's theorem \cite{myer41}, if the scalar curvature is positive then then $M$ is compact and has a finite fundamental group. Consequently, $b_1=0$. In dimension four, there are not many compact examples. In fact, all known examples which are simply connected with non-negative sectional curvature are already listed in Section \ref{prem}.

Also, it is noted that, since $\Rc-\frac{\SS}{4}g\equiv 0$, equation (\ref{curvdec}) implies that, for any plane $P$,
\[\KK(P)=\KK(P^\perp).\]
In particular, 
\[\KK^\perp (P)=\KK(P).\]

Then Corollary \ref{Eupperbound} is immediate.
\begin{proof}(\textbf{of Corollary \ref{Eupperbound}}) By Theorem \ref{selfdualtheorem}, $(M,g)$ must be a self-dual or anti-self-dual Einstein manifold. Then applying Hitchin's classification \cite[Theorem 13.30] {Besse} yields the desired conclusion.

\end{proof}

The proof of Theorem \ref{Epositiveintersectionthm} follows from a different argument. We first recall the following useful results. The first lemma says that an upper bound actually leads to a lower bound, which is better than the a priori bound coming from the algebraic relations.   
\begin{lemma}\label{lowerfromupper}(\cite[Lemma 3.3]{caotran4})
	\label{Kupper}
	Suppose that $\KK_{\max}=\alpha\leq 1$, then we have:
	\begin{align*}
	\KK_{\min} & \geq \frac{1}{28}(15-8\alpha-\sqrt{3}\sqrt{96\alpha^2-80\alpha+19}).
	\end{align*} 
\end{lemma}
The next result gives an estimate on the Euler characteristic when the sectional curvature is bounded above and below. 
\begin{lemma}
	\label{eulerestimate}(\cite[Corollary 3.1]{caotran4})Suppose that $\beta\leq \KK_{\min}\leq \KK_{\max} \leq \alpha$, then 
	\[8\pi^2 \chi(M) \leq \Big(8(\alpha^2-(1-\beta)(\alpha+\beta))+\frac{10}{3}\Big)\text{Vol}(M).\]
\end{lemma}

Moreover, there is an integral gap theorem for the self-dual Weyl curvature. 
\begin{theorem} \label{l2Wlower}(\cite[Theorem 1]{gl99}) 
	Let $(M,g)$ be a compact oriented Einstein 4-manifold with positve scalar curvature and $\WW^+\not\equiv 0$. Then,
	\[\int_{M} |\WW^+|^2 d\mu \geq \int_{M}\frac{\SS^2}{24} d\mu,\]
	with equality iff $\nabla \WW^+\equiv 0$.  
\end{theorem}
We are now ready to prove Theorem \ref{Epositiveintersectionthm}. 
\begin{proof} (\textbf{of Theorem \ref{Epositiveintersectionthm}}.)
By Lemma \ref{lowerfromupper}, $\KK\leq 1$ implies 
\[\KK\geq \frac{1}{28}(7-\sqrt{105}):=\beta.\]
Applying Lemma \ref{eulerestimate} then yields,
\begin{equation}
\label{s5eq1}
8\pi^2 \chi(M)\leq (8\beta^2+\frac{10}{3})\text{Vol}(M).\end{equation}
Combining the identities for Euler characteristic  (\ref{Euler}) and signature (\ref{signature}) leads to
\begin{align*}
(2\chi-3\tau)(M) &= \frac{1}{4\pi^2}\int_M (2|\WW^-|^2+\frac{\SS^2}{24})d\mu.
\end{align*} 
If $\WW^-\not\equiv 0$, then, by Theorem \ref{l2Wlower}(reversing the orientation of $M$ interchanges $\WW^+$ and $\WW^-$), we have
\begin{equation}\label{s5eq2}
(2\chi-3\tau)(M)\geq \frac{3}{4\pi^2}\int_{M}\frac{\SS^2}{24}d\mu=\frac{1}{2\pi^2}\text{Vol}(M).
\end{equation}
Combining equation (\ref{s5eq1}) with equation (\ref{s5eq2}) then yields
\begin{align*}
(2\chi-3\tau)(M) &\geq \frac{4\chi(M)}{8\beta^2+\frac{10}{3}},\\
(2-\frac{4}{8\beta^2+\frac{10}{3}})\chi(M) &\geq 3\tau(M).
\end{align*}
By reversing the direction, we obtain 
\[(2-\frac{4}{8\beta^2+\frac{10}{3}})\chi(M) \geq -3\tau(M).\]
Then we have,
\begin{align*}
(2-\frac{4}{8\beta^2+\frac{10}{3}})\chi(M) &\geq 3|\tau(M)|=3b_{+}\\
&\geq 2+b_{+}=\chi(M).
\end{align*}
Therefore, 
\begin{align*}
(2-\frac{4}{8\beta^2+\frac{10}{3}}) &\geq 1,\\
8\beta^2+\frac{10}{3} &\geq 4.
\end{align*}
The last inequality is a contradiction to the definition of $\beta$. 
\end{proof}
\def\cprime{$'$}
\bibliographystyle{plain}
\bibliography{bioEin}
\end{document}